\documentclass[11pt]{amsart}
\usepackage{amssymb,latexsym,times,amsmath,amsthm, graphicx, epsfig}

\newtheorem{theorem}{Theorem}
\newtheorem{conjecture}{Conjecture}

\newtheorem{lemma}{Lemma}

\begin{document}

\title{Reciprocals of false theta functions}

\author{William J. Keith}

\address{Department of Mathematical Sciences\newline
Michigan Technological University\newline
Houghton, Michigan 49931, U.S.A.\newline
email: wjkeith@mtu.edu\newline
ORCID: 0000-0002-5109-0216}

\begin{abstract} We investigate reciprocals of false theta functions, producing results such as congruences, simple asymptotic bounds, and combinatorial identities.  Of particular interest is a connection between $1/\Psi(-q^2,q)$ and the truncated pentagonal number theorem of Andrews and Merca.  We record a useful dissection identity analogous to the known theta function dissection.
\end{abstract}

\maketitle

\emph{Dedicated to George Andrews and Bruce Berndt for their $85^{th}$ birthdays}

\section{Introduction}

Ramanujan's \emph{general theta function} is the bilateral sum $$f(a,b) := \sum_{n=-\infty}^\infty a^{\binom{n+1}{2}} b^{\binom{n}{2}}.$$ It is symmetric in its arguments and famously satisfies the Jacobi Triple Product identity $$f(a,b) = (-a; ab)_\infty (-b; ab)_\infty (ab;ab)_\infty $$ where we are using the Pochhammer symbol $$(a;b)_n = \prod_{k=0}^{n-1} (1-a b^k) \quad , \quad (a;b)_\infty = \lim_{n \rightarrow \infty} (a;b)_n.$$  An important example is $$f(-q,-q^2) = (q;q)_\infty = 1 - q - q^2 + q^5 + q^7 - q^{12} - q^{15} + \dots ,$$ \noindent for which the reciprocal is the generating function for the number of partitions of an integer $n$, $$\frac{1}{(q;q)_\infty} = \sum_{n=0}^\infty p(n) q^n.$$

Ramanujan defined and studied several variations of theta functions.  In his last letter to Hardy \cite{LastLetter} he introduced the \emph{mock theta functions}, saying of them, \emph{I discovered very interesting functions recently which I call ``Mock'' $\theta$-functions.  Unlike the ``False'' $\theta$-functions (partially studied by Rogers), they enter into mathematics as beautifully as the ordinary theta functions....}

The \emph{false theta functions} to which Ramanujan refers are a sign perturbation of theta functions, which are defined thus: $$\Psi(a,b) := \sum_{n=0}^\infty a^{\binom{n+1}{2}} b^{\binom{n}{2}} - \sum_{n=-\infty}^{-1} a^{\binom{n+1}{2}} b^{\binom{n}{2}} .$$

Considered as functions of complex variables, all of these objects labeled as some variation of ``theta function'' generally have in common some well-behaved convergence criterion on or near the unit circle, the precise details of which are not required for this paper.  In Zwegers' Ph.D. thesis \cite{Zwegers} he showed that the mock theta functions are the holomorphic parts of weight $1/2$ harmonic Maass forms.  They have extraordinary symmetries giving rise to much work and many beautiful results.  For a recent survey see Gordon and McIntosh \cite{GorMc}; a shorter and more accessible introduction for a more general audience was written by Duke \cite{Duke}.

False theta functions, as one might expect from the sign change, are less symmetric.  Bringmann and Milas \cite{BringMilas}, and Goswami and Osburn \cite{GoswamiOsburn}, showed that these are examples of the recently named class of \emph{quantum modular forms}, which have looser convergence criteria.  Following Ramanujan's judgment, one might therefore have less hope of beautiful theorems associated to them.  In particular, congruences of the form $a_f(An+B) \equiv 0 \pmod{m}$ for the coefficients of a power series $f(q) = \sum_{n=n_0}^\infty a_f(n)q^n$ often arise in partition theory essentially as a consequence of the transformation properties satisfied by modular forms.

Still, much like the Rogers-Ramanujan identities, a function that drew the attention of two such great mathematicians should have plenty to recommend investigation.  Both Rogers and Ramanujan produced several interesting results concerning false theta functions.  Ramanujan's Lost Notebook features the example \begin{equation}\label{RLNPsi}\Psi(q^3,q) = \sum_{n=0}^\infty \frac{(q;q^2)_n q^n}{(-q;q^2)_{n+1}}.\end{equation}

The goal of this paper is to consider some specific cases of false theta functions, in particular their reciprocals.  Our focus will be on $$\frac{1}{\Psi(-q^t,q)} =: \sum_{n=0}^\infty c_t(n) q^n.$$

We will establish several results which appear to be fairly untrodden ground in the literature: several congruences, an exponential asymptotic, and some side results, including an interesting connection to the \emph{truncated pentagonal number theorem} recently studied by Andrews and Merca.  All of our techniques are relatively elementary, leaving many potential avenues of investigation for interested readers.  The final section suggests a number of questions.

\section{Background and notations}

We collect here a few notations and common identities that will be useful in the paper.

When we say of power series $f(q) := \sum_{n=n_1} a(n) q^n$ and $g(q) := \sum_{n=n_2} b(n) q^n$ that $f(q) \equiv_m g(q)$, we mean that $a(n) \equiv b(n) \pmod{m}$ for all $n \in \mathbb{Z}$.

By Kummer's theorem for binomial coefficients, it holds for primes $p$ and integers that $$\left( \sum_{n=n_0}^\infty f(n) q^n \right)^{p^k} \equiv_{p^k}  \left( \sum_{n=n_0}^\infty f(n) q^{pn} \right)^{p^{k-1}}.$$  The most common use we will make of this fact will be for squaring modulo 2, in that $\left( \sum_{n=0}^\infty f(n) q^n \right)^2 \equiv_2  \sum_{n=0}^\infty f(n) q^{2n}$.  We will frequently use this fact, and the fact that $f(q) \equiv_2 -f(q)$ or indeed any collection of individual sign changes on the coefficients, without repeated comment.

We use throughout the standard notations for the $q$-Pochhammer symbol, $$(a;b)_n = (1-a) (1-ab) \dots (1-ab^{n-1}), \quad \quad (a;b)_\infty = \lim_{n \rightarrow \infty} (a;b)_n .$$  We will also use $f_k := (q^k;q^k)_\infty$.

Two further congruences which will be useful to us include the classical fact due to Jacobi \cite{oeis} that \begin{equation}\label{cubeeq} f_1^3 = \sum_{n=0}^\infty (-1)^n (2n+1) q^{\binom{n+1}{2}} \equiv_2 \sum_{n=0}^\infty q^{\binom{n+1}{2}} \end{equation} and (\cite{BBG}, p. 301) \begin{equation}\label{f1f5} f_1 f_5 \equiv_2 f_1^6 + q f_5^6 .\end{equation}

\section{Congruences}\label{seccon}

Recall that $$1/\Psi(-q^5,q) =: \sum_{n=0}^\infty c_5(n) q^n.$$

We first show that $c_5(n)$ possesses many even arithmetic progressions.

\begin{theorem} $$\sum_{n=0}^\infty c_5(2n+1) q^n \equiv_2 \sum_{n \in \mathbb{Z}}^\infty q^{n(3n-2)}$$ and therefore, by arithmetic properties of the exponents, for any prime $p>3$ and all $n \geq 0$, it holds that $$c_5(2(pn+j)+1) \equiv 0 \pmod{2}$$ if $3^{-1}(j+3^{-1})$ is not a quadratic residue mod $p$.  For instance, it holds that $$c_2(10n+5) \equiv 0 \pmod{2} \quad \text{ and } \quad c_2(10+9) \equiv 0 \pmod{2}.$$ It also holds that for all $n \geq 0$, $$c_5(8n+5) \equiv 0 \pmod{2}$$ along with many similar congruences in residue classes for powers of 2.
\end{theorem}

\begin{proof}  Employing the identities from the previous section, we have the following.

$$\Psi(-q^5,q) \equiv_2 f(q^5,q) = (q;q^6)_\infty (q^5;q^6)_\infty (q^6;q^6)_\infty \equiv_2 \frac{{(q^3;q^3)_\infty}^3}{(q;q)_\infty}.$$

Hence by the identity (Xia and Yao \cite{XiaYao}, Theorem 3.7 clauses 3.63 and 3.64), $$\frac{f_3^3}{f_1} \equiv_2 f_1^8 + q \frac{f_3^{12}}{f_1^4},$$ we have the congruence 

\begin{multline*} \frac{1}{\Psi(-q^5,q)} =: \sum_{n=0}^\infty a(n) q^n \equiv_2 \frac{(q;q)_\infty}{{(q^3;q^3)_\infty}^3} \equiv_2 \frac{{(q;q)_\infty}^{10}}{{(q^3;q^3)_\infty}^6} + q \frac{{(q^3;q^3)_\infty}^6}{{(q;q)_\infty}^2}.
\end{multline*}

Thus $$\sum_{n=0}^\infty c_5(2n+1) q^n \equiv_2 \frac{(q^3;q^3)_\infty^3}{(q;q)_\infty} \equiv \sum_{n \in \mathbb{Z}}^\infty q^{n(3n-2)}.$$ But the middle term is the generating function for the well-understood 3-\emph{core partitions}, and the final equivalence follows by Theorem 7 of Robbins \cite{Robbins}.

Now we note that the equation $mp+j = n(3n-2) = 3n^2 - 2n$ implies the following congruences, where inverses are taken mod $p$ for $p$ a prime larger than 3: \begin{align*}j &\equiv 3n^2 - 2n \pmod{p} \\
3^{-1} j &\equiv n^2 - 2 \cdot 3^{-1} n \pmod{p} \\
3^{-1}(j+3^{-1}) &\equiv (n - 3^{-1})^2 \pmod{p}.
\end{align*}

Hence if $3^{-1}(j+3^{-1})$ is not a quadratic residue modulo $p$, then it must hold that $c_5(2(pn+j)+1) \equiv 0 \pmod{2}$.  For the instance given in the theorem, $3n^2 - 2n$ takes on the values 0, 1, and 3 modulo 5, and hence $c_2(2(5n+j)+1) \equiv 0 \pmod{2}$ for $j = 2$ and $j=4$.  For the prime 2, we note that $3n^2 - 2n$ is only 0 or 1 modulo 4, and so $2n+1$ must miss $2(4k+2)+1$ for all integer $k$, along with all the other congruences modulo powers of 2 that can be shown by the same analysis.

\end{proof}

Next, by considerably more involved means,

\begin{theorem} $c_5(32n+31) \equiv 0 \pmod{4}.$
\end{theorem}

\begin{proof} We require several dissections and congruences:

\begin{align*}
\Psi(-q^5,q) &= \Psi(-q^{16},-q^8) - q \Psi(-q^{20},-q^4) =: A(q^8) - q B(q^4) \\
B(q) &= \Psi(q^{16},q^8) + q \Psi(q^{20},q^4) =: F(q^8) + q G(q^4) \\
G(q) &= F(q^8) - q G(q^4) \\
\left(\sum_{n=0}^\infty a(n) q^n \right)^4 &\equiv_4 \left(\sum_{n=0}^\infty a(n) q^{2n} \right)^2 \\
A(q)^2 &\equiv_4 F(q)^2 \equiv_4 H(q^2) + 2q I(q^2) \text{ for some } H, I \\
B(q)^2 &\equiv_4 G(q)^2.
\end{align*}

The latter two statements are consequences of the fact that signs might matter only on cross terms, but these appear multiplied by 2, and $2 \equiv -2 \pmod{4}$.

The first three clauses are cases of a more general dissection of false theta functions of this type.  The dissection for theta functions is known; it is for instance a corollary of the more general identity of \cite{Bailey}, p. 220, Eq. (4.1)) given below. $$f(z,q/z) = f(z^2 q, z^{-2} q^3) + z f(z^2 q^3, z^{-2} q).$$ In this if we substitute $q \rightarrow q^{a+b}, z \rightarrow q^b$ and employ the symmetry of $f(r,s)$ in its arguments, we obtain $$f(q^a, q^b) =  f(q^{a+3b}, q^{3a+b}) + q^b f(q^{3a+5b}, q^{a-b}).$$  When both of $a$ and $b$ are odd, this is an even-odd dissection, i.e. the first term has only even powers of $q$ and the latter only odd.

We state here the analogous false theta dissection, which appears to be new in the literature.  (For instance, it does not appear in Sills \cite{Sills}, a fairly extensive recent survey of related work.)  It is not symmetric in its arguments, so there are two cases.

\begin{lemma}\label{evenoddfalsetheta}
Let $\epsilon \in \{ 1, -1 \}$, and let $ a > b \geq 0$ be integers.  Then $$\Psi(\pm q^a, \pm \epsilon q^b) = \Psi(\epsilon q^{3a+b}, \epsilon q^{a+3b}) \pm (-\epsilon) q^b \Psi(\epsilon q^{3a+5b}, \epsilon q^{a-b}).$$
If instead $b > a \geq 0$, then $$\Psi(\pm q^a, \pm \epsilon q^b) = \Psi(\epsilon q^{3a+b}, \epsilon q^{a+3b}) \pm q^a \Psi(\epsilon q^{5a+3b}, \epsilon q^{b-a}).$$
\end{lemma}

Here the $\pm$ signs are not independent.  That is to say, the arguments of the functions on the right hand side are all positive if the signs of the arguments on the left hand side match, and negative if they mismatch; and the sign between the terms on the right hand side is either  the opposite of the sign on the second argument on the left hand side if $a > b$, or the same as that of the first argument if $ a < b$.

\begin{proof} Expand the left hand side from $$\sum_{n=0}^\infty q^{a \binom{n+1}{2}} q^{b \binom{n}{2}} - \sum_{n=-\infty}^{-1} q^{a \binom{n+1}{2}} q^{b \binom{n}{2}}$$ by separating the indices by their residue modulo 4.  Expand the binomials and observe their parities, using the fact that $\binom{n+1}{2} \equiv 0 \pmod{2}$ iff $n \equiv 0, 3 \pmod{4}$.  Gather the appropriate terms by signs.  Sometimes it will be necessary to reverse signs and shift an index by 1.

For example, the most immediately relevant case of the lemma is $$\Psi(-q^a,q^b) = \Psi(-q^{3a+b},-q^{a+3b})-q^b \Psi(-q^{3a+5b},-q^{a-b})$$ with $a$ and $b$ both odd, and $a > b$. We expand thus:

\begin{align*}
\Psi(-q^a,q^b) &= \sum_{{n \geq 0} \atop {n = 4i}} q^{a(8i^2+2i)+b(8i^2-2i)} - \sum_{{n \geq 0} \atop {n = 4i+1}} q^{a(8i^2+6i+1)+b(8i^2+2i)} \\
 & \quad - \sum_{{n \geq 0} \atop {n = 4i+2}} q^{a(8i^2+10i+3)+b(8i^2+6i+1)} + \sum_{{n \geq 0} \atop {n = 4i+3}} q^{a(8i^2+14i+6)+b(8i^2+10i+3)}\\
&- \sum_{{n < 0} \atop {n = 4i}} (\dots) + \sum_{{n < 0} \atop {n = 4i+1}} (\dots) + \sum_{{n < 0} \atop {n = 4i+2}} (\dots) - \sum_{{n < 0} \atop {n = 4i+3}} (\dots)  \\
&= \sum_{i \geq 0} (-q^{3a+b})^{\binom{i+1}{2}}(-q^{a+3b})^{\binom{i}{2}} - \sum_{i < 0} (-q^{3a+b})^{\binom{i+1}{2}}(-q^{a+3b})^{\binom{i}{2}} \\
 & \quad - q^b \left(\sum_{j \geq 0} (-q^{3a+5b})^{\binom{j+1}{2}} (-q^{a-b})^{\binom{j}{2}} -\sum_{j < 0} (-q^{3a+5b})^{\binom{j+1}{2}} (-q^{a-b})^{\binom{j}{2}} \right).
\end{align*}

In the last line the first, third, fifth, and seventh terms of the first line became the first and second summands, using such identities as $$(3a+b)\binom{2j+2}{2} + (a+3b)\binom{2j+1}{2} = a(8j^2+10j+3) + b(8j^2+6j+1).$$  The remaining terms became the summands under the factor $q^b$ after the substitutions $i \rightarrow -i, j = i-1$.  Other cases are minor variations on this proof.  \end{proof}

Then: 

\begin{multline*}
\frac{1}{\Psi(-q^5,q)} = \frac{1}{A(q^8) - q B(q^4)} \\
= \frac{1}{A(q^8)^{32} - q^{32} B(q^4)^{32}} \left[ \left((A(q^8)+q B(q^4)\right) \right. \\
 \times \left. \left((A(q^8)^2+q^2 B(q^4)^2 \right) \dots  \left((A(q^8)^{16}+q^{16} B(q^4)^{16} \right) \right] \\
\equiv :_4 \frac{1}{D(q^{32})} \left[ \left(A(q^{32})^4 + q^{16} B(q^{32})^2\right) \left(A(q^{32})^2 + q^{8} B(q^{16})^2 \right) \right. \\ 
\times \left. \left(A(q^{16})^2 + q^{4} B(q^{8})^2 \right) \left(A(q^{8})^2 + q^{2} B(q^{4})^2 \right) (A(q^8) + q B(q^4))  \right] .
\end{multline*}

Now the denominator is nonzero modulo 4, and is a function of $q^{32}$, so the theorem is true if and only if the coefficients of $q^{32n+31}$ in the numerator all cancel modulo 4.

Begin by noticing every function is at least in terms of $q^4$, so we cannot avoid selecting the $q B(q^4)$ and $q^2 B(q^4)^2$ terms from the product.  Removing a $q^{3}$ and reducing powers by a factor of 4, it remains to examine the $8n+7$ terms in 

$$B(q)^3 \left(A(q^{8})^4 + q^{4} B(q^{8})^2\right) \left(A(q^{8})^2 + q^{8} B(q^{4})^2 \right) \left(A(q^{4})^2 + q B(q^{2})^2 \right).$$

We expand out any $A$ or $B$ that aren't already in powers of $q^8$, and then expand the product.  We select out the terms that contribute to $8n+7$, and begin using identities to cancel terms.  Eventually, everything disappears mod 4.  

The actual cancellations would be tedious to display, so we give brief examples of the two most common types.

After expansion, the term with highest ``external'' power of $q$ is \begin{multline*} \left(-2 A^4(q^8)F(q^8)F(q^{16}) G(q^8)G^4(q^{16}) \right. \\ \left. +2A^2(q^8)B^2(q^8)G^3(q^{16})I(q^8)-2B^2(q^8)F^2(q^8) G^3(q^{16}) I(q^8) \right) q^{16}.\end{multline*}

Now \begin{multline*} A^2(q^8)B^2(q^8)G^3(q^{16})I(q^8) - B^2(q^8)F^2(q^8) G^3(q^{16}) I(q^8) \\
= (A^2(q^8) - F^2(q^8)) (B^2(q^8)G^3(q^{16})I(q^8)) \equiv_4 0.\end{multline*}

Thus this term becomes $$\left(-2 A^4(q^8)F(q^8)F(q^{16}) G(q^8)G^4(q^{16}) \right) q^{16}.$$

In a later step, we employ the fact that $F(q^{128}) + q^{16} G(q^{64}) = B(q^{16})$ to gather two terms: \begin{multline*} -2 (A^4(q^8) F(q^8) F(q^{16}) G(q^8)) ( F(q^{128}) + q^{16} G(q^{64}) ) \\ \equiv_4 -2 A^4(q^8) F(q^8) F(q^{16}) G(q^8) B(q^{16}).\end{multline*}

By dint of repeated uses of such properties, eventually we show that all terms cancel mod 4.

\end{proof}

For $c_9$ we have an infinite class of congruences modulo 2 within the residue class $8n+4$.

\begin{theorem} It holds that $\sum_{n=0}^\infty c_9(8n+4) q^n \equiv_2 \Psi(-q^{14},-q^6)$.  As a result, $c_9(8n+4)$ possesses many even progressions of the form $c_9(An+B) \equiv 0 \pmod{2}$ for all $n$, including $$(A,B) \in \{ (16,12), (24,12), (56,20), (56,28), (56,44) , \dots \}.$$  More generally, $c_9(8n+4) \equiv 0 \pmod{2}$ if $n$ cannot be represented in the form $n = 10k^2 - 4k$ for integer $k$.
\end{theorem}

\begin{proof} By Lemma \ref{evenoddfalsetheta}, we have that $$\Psi(-q^9,q) = \Psi(-q^{28},-q^{12}) - \Psi(-q^{32},-q^8) := A(q^4) - q B(q^8).$$  Likewise $$A(q) = \Psi(-q^7,-q^3) = \Psi(q^{24},q^{16}) + q^3 \Psi(q^{36},q^4) := C(q^8) + q^3 D(q^4).$$

We now observe \begin{multline*}\frac{1}{\Psi(-q^9,q)} \\ = \frac{(A(q^4) + q B(q^8))(A(q^8) + q^2 B(q^{16}))(A(q^{16}) + q^4 B(q^{32}))}{(A(q^{32}) - q^{8} B(q^{64}))} \\ \equiv_2 \frac{(C(q^{32}) + q^{12} D(q^{16}) + q B(q^8))(A(q^8) + q^2 B(q^{16}))(A(q^{16}) + q^4 B(q^{32}))}{(A(q^{32}) - q^{8} B(q^{64}))} \end{multline*}

Extracting terms $q^{8n+4}$, we find that: 

$$\sum_{n=0}^\infty c_9(8n+4) q^n =  \frac{A(q)}{A(q^{4}) - q B(q^{8})} \left[ C(q^4) B(q^4) + q D(q^2) A(q^2) \right].$$

Hence the theorem will be proved if we can show the identity $$C(q^4) B(q^4) + q D(q^2) A(q^2) \equiv_2 A(q) D(q).$$

Expressing this via the Jacobi Triple Product, we find that this identity is equivalent to \begin{multline*}(q,q^9,q^{10};q^{10})_\infty \equiv_2 (q^{12},q^8,q^{20};q^{20})_\infty (q^{16},q^4,q^{20};q^{20})_\infty \\ + q (q^{18},q^2,q^{20};q^{20})_\infty (q^{14},q^6,q^{20};q^{20})_\infty .\end{multline*}

Multiplying various terms by instances of $\frac{(q;q^{2})_\infty}{(q;q^{2})_\infty}$, $\frac{(q^5;q^{10})_\infty}{(q^5;q^{10})_\infty}$ and similar factors, we can rewrite this an an identity of eta-products, equivalent to $$f_4^2 f_5^2 f_{20} + q f_1^2 f_{10}^6 \equiv_2 f_1 f_2 f_5 f_{10}^3.$$

Clearing common factors and employing the equivalence $f_1^2 \equiv_2 f_2$, we find that the above claim is equivalent to $$f_1 f_5 \equiv_2 f_1^6 + q f_5^6,$$ \noindent the identity given by congruences (\ref{f1f5}) and (\ref{cubeeq}), and the desired congruence is proved. 

To complete the theorem we now observe that $\Psi(-q^{6},-q^{14})$ features terms $q^n$ with powers $$n = 6 \binom{k+1}{2} + 14 \binom{k}{2} = 10k^2 - 4k.$$  Quick calculations now yield the many arithmetic progressions entirely avoided by the series.  (The theorem could also have been stated in terms of the quadratic residues avoided by $10k^2 - 4k$ modulo $m$ among $8(mn+j)+4$.)

\end{proof}

\section{Asymptotics}

It is interesting to compare the false theta function $\Psi(-q^2,q)$ and the function $(q;q)_\infty$, as the first only differs by signs from the second: 

\begin{align*}
\psi(-q^2,q) &= 1 - q - q^2 + q^5 - q^7 + q^{12} + q^{15} - q^{22} + q^{26} - \dots \\
(q;q)_\infty &= 1 - q - q^2 + q^5 + q^7 - q^{12} - q^{15} + q^{22} + q^{26} - \dots
\end{align*}

Nonzero coefficients occur when the power on $q$ is one of the \emph{pentagonal numbers} given by $n = \frac{m}{2}(3m-1)$ for $m \in \mathbb{Z}$.  The sign pattern for $(q;q)_\infty$ is $-,-,+,+$, while the signs in the false theta function have period 8, given by $-,-,+,-,+,+,-,+$.  This claim is easily proved by simple examination, by residue class modulo 4, of the binomial numbers appearing in the original definition of the false theta function and their signs.

Recall that the partition function $p(n)$ counting the number of partitions of $n$ has generating function $$\frac{1}{(q;q)_\infty} = \sum_{n=0}^\infty p(n) q^n.$$  The partition function was shown by Hardy and Ramanujan \cite{HR} to have subexponential growth, with main asymptotic $$p(n) \approx \frac{1}{4n\sqrt{3}} e^{\pi \sqrt{2n/3}}.$$  If we consider $$\frac{1}{\Psi(-q^2,q)},$$ then heuristically we might expect that ``earlier negative signs make larger coefficients.''  In fact, growth is now exponential: 

\begin{theorem} The coefficients $c_2(n)$ grow exponentially with base $b$ in the interval $1.53623 < b < 1.54522$.
\end{theorem}

\begin{proof}  We begin by observing that $c_2(n)$ can be written as the unique solution of an infinite recurrence, given by $c_2(n) = 0$ for $n < 0$, $c_2(0) = 1$, and for $n>0$, \begin{multline*}c_2(n) = c_2(n-1) + c_2(n-2) - c_2(n-5) + c_2(n-7) \\ - c_2(n-12) - c_2(n-15) + c_2(n-22) - c_2(n-26) + \dots .\end{multline*}

Signs thereafter repeat periodically.

Since the first few values of $c_2(n)$ for $0 \leq n \leq 27$ are $1, 1, 2, 3, 5, 7, 11, 17, \dots$ and form an increasing sequence, and thereafter terms in the recurrence are added before any equal number are subtracted, it immediately follows that $c_2(n)$ is strictly increasing for $n>0$.

We next observe that this periodicity of signs means that the values of $c_2(n)$ are bounded below by the values of a sequence $b(n)$ satisfying the finite linear recurrence given by the same initial conditions, and truncating the recurrence at degree 26: 

\begin{multline*}b(n) = b(n-1) + b(n-2) - b(n-5) + b(n-7) \\ - b(n-12) - b(n-15) + b(n-22) - b(n-26).\end{multline*}

Likewise, the sequence is bounded above by the values $a(n)$ satisfying the same initial conditions and the finite linear recurrence $$a(n) = a(n-1) + a(n-2) - a(n-5) + a(n-7).$$

Both claims require a (multidimensional but straightforward) inductive argument hypothesizing that all differences at sufficient distance are similarly bounded, i.e. for all $n \geq 0$, and for $1 \leq z \leq 7$ in the case of $a$ and $1 \leq z \leq 26$ in the case of $b$, it holds that $$b(n) - b(n-z) \leq c_2(n) - c_2(n-z) \leq a(n) - a(n-z).$$

As a sample of the argument, notice that \begin{align*} a(n) &= a(n-1) + a(n-2) - a(n-5) + a(n-7) \\
c_2(n) &= c_2(n-1) + c_2(n-2) - c_2(n-5) + c_2(n-7) \\ &\phantom{=} - c_2(n-12) - \dots \\
a(n) - a(n-1) &=  a(n-2) - a(n-5) + a(n-7) \\
c_2(n) - c_2(n-1) &= c_2(n-2) - c_2(n-5) + c_2(n-7) - c_2(n-12) - \dots .
\end{align*}

Thus, noting that $c_2$ is a positive, increasing sequence, both claims $a(n) \geq c_2(n)$ and $a(n) - a(n-1) \geq c_2(n) - c_2(n-1)$ hold if it is inductively true that $a(n^\prime) \geq c_2(n^\prime)$ for all $n^\prime < n$ and that $a(n^\prime) - a(n^\prime-3) \geq  c_2(n^\prime) - c_2(n^\prime-3)$ for all $n^\prime < n$.  To establish the latter inductively we observe that 

\begin{align*}
a(n) - a(n-3) &= 2a(n-2) - a(n-5) - a(n-6) + a(n-7) + a(n-8) \\
&= 2(a(n-2) - a(n-5)) + a(n-5) - a(n-6) + (\dots) \\
c_2(n) - c_2(n-3) &= 2(c_2(n-2) - c_2(n-5)) + (c_2(n-5) - c_2(n-6)) + (\dots) .
\end{align*}

Now we note that inductively it holds that $a(n-2) - a(n-5) \geq c_2(n-2) - c_2(n-5)$, that $a(n-5) - a(n-6) \geq c_2(n-5) \geq c_2(n-6)$, and the remaining elided portions are bounded directly by $a(n) \geq c_2(n)$.  Similar arguments work for $b(n)$ though more are required.

We now use a standard theorem in the theory of finite linear recurrences, which is that the solutions can be represented as a linear combination of exponentials in the roots of the characteristic equation.  In this case, we use a numerical algebra package to confirm that the largest root of $$x^7 - x^6 - x^5 + x^2 - 1$$ is approximately $x = 1.54522$, and the largest root of $$x^{26} - x^{25} - x^{24} + x^{21} - x^{19} + x^{14} + x^{11} - x^4 + 1$$ is approximately $x = 1.53623$, and the proof is concluded. \end{proof}

Exploration of the relations between $1/\Psi(-q^2,q)$ and the partition function yields a tangential but interesting connection with Andrews and Merca's study \cite{TruncPent} of the truncated pentagonal number theorem.

We observe the following:

\begin{align*}
\frac{1}{\Psi(-q^2,q)} &= \frac{1}{(q;q)_\infty -2q^7 + 2q^{12} +2q^{15} -2q^{22} - \dots} \\
&= \frac{1}{(q;q)_\infty} \left(\frac{1}{1-\frac{2}{(q;q)_\infty}(q^7 -q^{12} -q^{15} +q^{22} + \dots ) } \right) .
\end{align*}

The function $\frac{q^7 -q^{12} -q^{15} +q^{22}}{(q;q)_\infty}$ is exactly the generating function for $p(n-7) - p(n-12) - p(n-15) + p(n-22)$.  Andrews and Merca in \cite{TruncPent} studied the quantities \phantom{.}

\noindent \begin{tabular}{l}
$p(n) - p(n-1)$ \\
$p(n) - p(n-1) - p(n-2) + p(n-5)$ \\
$p(n) - p(n-1) - p(n-2) + p(n-5) + p(n-7) - p(n-12)$ \\
$p(n) - p(n-1) - p(n-2) + p(n-5) + p(n-7) - p(n-12) - p(n-15) + p(n-22)$ \\
\dots
\end{tabular}

\phantom{.}

\noindent which arise by truncating at an even number of terms the recurrence for the partition numbers given by Euler's pentagonal number theorem, that $p(0) = 1$ and for $n>0$, $$p(n) - p(n-1) - p(n-2) + p(n-5) + p(n-7) - \dots = 0.$$

They showed that these are, with alternating signs, $M_1$, $M_2$, $M_3$, $M_4$, \dots, where $M_k$ is the number of partitions of $n$ in which $k$ is the least integer not appearing as a part (the \emph{mex}), and there are more parts larger than $k$ than there are smaller than $k$.  These have the generating function $${\mathcal{M}}_k = \sum_{n=0}^\infty M_k(n) q^n = \sum_{n=k}^\infty \frac{q^{\binom{k}{2}+(k+1)n}}{(q;q)_n} \left[ {{n-1} \atop {k-1}} \right]_q.$$

But then by taking differences and grouping every four terms we have

$$1/\Psi(-q^2,q) = \frac{1}{(q;q)_\infty} \left( \frac{1}{1 - 2({\mathcal{M}}_2 - {\mathcal{M}}_4 + {\mathcal{M}}_6 - {\mathcal{M}}_8 + \dots)} \right).$$

This gives the following identity relating the false theta function and $(q;q)_\infty$:

\begin{theorem}\label{TPNThm}\begin{align*}\Psi(-q^2,q) &= (q;q)_\infty \left(1 - 2({\mathcal{M}}_2 - {\mathcal{M}}_4 + {\mathcal{M}}_6 - {\mathcal{M}}_8 + \dots) \right) \\
 &= (q;q)_\infty \left(1 - 2 \sum_{j=1}^\infty (-1)^{j+1} \sum_{n=2j}^\infty \frac{q^{\binom{2j}{2}+(2j+1)n}}{(q;q)_n} \left[ {{n-1} \atop {2j-1}} \right]_q \right).
\end{align*}
\end{theorem}

We remark that since $\frac{\Psi(-q^2,-q)-1}{(q;q)_\infty}$ is the generating function for partitions with rank 0, the theorem above displays quite clearly the combinatorial fact that the number of partitions with nonzero rank is even, which is normally seen as a consequence of conjugation negating the rank.


\section{Further Questions}

We conclude with some potential directions of further investigation.




The congruences proved in Section \ref{seccon} are exemplars of what appear to be a larger class,  infinite but with several restrictions.  For $c_t(An+B)$, we have the following behaviors suggested by empirical calculation.

\begin{conjecture} In addition to the infinite classes of congruences modulo 2 proven for $c_5$ and $c_9$ above, the following congruences appear to hold modulo 2:

\phantom{.}

\begin{tabular}{cllll}
$c_9$ : & $36n+14$, & $196n+(54,166,194)$, & $\dots$ &  \\
$c_{13}$ : & $32n+23$, & $64n+63$, & $72n+(15,21,39,69)$, & $\dots$ \\
$c_{17}$ : & $128n+80$, & $\dots$ & & 
\end{tabular}

\end{conjecture}

\begin{conjecture} The following congruences appear to hold for $c_5$: 

\phantom{.}

\begin{tabular}{clll}
$0 \pmod{8}$ : & $32n+31$, & $128n+123$, & $512n+491$ \\
$0 \pmod{4}$ : & $64n+19$, & $256n+75$, & $196n+7j+5$
\end{tabular}

\noindent for $j \in \{ 2, 6, 10, 14, 15, 19, 22, 26, 27 \}$.

\end{conjecture}

One may notice that the only candidate congruences observed in experimentation are modulo powers of 2, and arise for $c_t$ with $t \equiv 1 \pmod{4}$.  While not a complete explanation, observe that, as we are dealing with $\Psi(-q^t,q)$, the dissection lemma conveniently yields not only an even-odd dissection but in fact a dissection modulo 4, which would seem to assist in permitting congruences to arise quickly.

While congruences with other moduli are not ruled out by experimentation, it would certainly be interesting if these turned out to be the entire class of allowable congruences for functions of the type $1/\Psi(-q^a,q)$.

The long list of cancellations necessary to prove the congruence theorems modulo 4 seem like a prime candidate for potential automation: only a small number of identities should be necessary and many of the cancellations are repetitive.  The expansion steps also seem to potentially be algorithmically recognizable.

Identities such as Ramanujan's (\ref{RLNPsi}) for $\Psi(q^3,q)$ and Theorem \ref{TPNThm} for $\Psi(-q^2,q)$ would certainly be of interest for reciprocals of false theta functions.

The asymptotics of reciprocals of false theta functions in general could be of interest.  The methods used here were crude but easily accessible; perhaps the exact asymptotic could be determined, and likewise others.  For instance, the series $1/\Psi(-q^3,q)$ seems to yield a sequence of coefficients with exponential ratio approximately 1.37.  We would certainly offer the general conjecture that follows.

\begin{conjecture} The coefficients of the series $1/\Psi(-q^a,q)$ grow exponentially.
\end{conjecture}

Does this hold for other classes of false theta function reciprocals?  What about other, related combinatorial objects?

For an extreme, consider $$\frac{1}{1-q-q^2-q^5-q^7-q^{12}-q^{15}-q^{22}-q^{26}-\dots}.$$

This is the generating function for compositions into pentagonal numbers.  The partition numbers are a count of the same set, with a sign weighting, as are the coefficients of $1/\Psi(-q^2,q)$, but only the partition number have the precise delicate balance that yields a subexponential growth rate.  (Indeed, what is the growth rate for this function?  It must exceed the golden ratio which yields the Fibonacci numbers, but be less than 2, the exponential base for unrestricted compositions.)

All sign variations of the triple product $$(q^a;q^{a+b})_\infty (q^b;q^{a+b})_\infty (q^{a+b};q^{a+b})_\infty ,$$ including the related false theta functions, form an equivalence class.  Among these, the original signs give the ordinary partitions into parts congruent to $a$, $b$, or 0 modulo $a+b$, and this yields a sequence of subexponential growth.  All other choices, experimentally, seem to have either exponential growth, or eventually oscillatory signs.  Are the original triple product functions really the only special case among this class, and if so, why?

Originally, it was hoped that that $M_{4k-2}(n) \geq M_{4k}(n)$ for all $k, n \geq 0$, which would have shown that the reciprocal $1/\Psi(-q^2,q)$ is the generating function for partitions, times something with strictly nonnegative coefficients.  The former claim turns out not to be true, though the final ratio $(q;q)_\infty/\Psi(-q^2,q)$ does appear to have strictly positive coefficients.  Can this be verified, and given a more intuitive combinatorial explanation?

\section{Declarations}

No funding was received to assist with the preparation of this manuscript.  The author declares that he has no competing financial interests.

\section{Acknowledgements}

The material in this paper was presented at Penn State at the Legacy of Ramanujan 2024 conference celebrating the 85th birthdays of George Andrews and Bruce Berndt.  The author cordially thanks the organizers for the invitation, and Profs. Andrews and Berndt for their masterful contributions to the field.

\end{document}